\documentclass[11pt]{article}
\usepackage{amsmath}
\usepackage{amssymb}
\usepackage{amsthm}
\usepackage{tikz} 
\usepackage{caption}
\usepackage{dsfont}
\usepackage{indentfirst}

\usepackage[colorlinks=true,linkcolor=blue,filecolor=red,
citecolor=webgreen]{hyperref}
\definecolor{webgreen}{rgb}{0,.5,0}

\def\C{{\mathds{C}}}

\def\N{{\mathds{N}}}

\def\1{{\bf 1}}

\def\lcm{\operatorname{lcm}}
\def\deg{\operatorname{deg}}

\newtheorem{theorem}{Theorem}

\newtheorem{lemma}[theorem]{Lemma}
\newtheorem{corollary}[theorem]{Corollary}
\newtheorem{example}[theorem]{Example}

\newtheorem{remark}[theorem]{Remark}

\begin{document}

\title{{\bf On the asymptotic density of $k$-tuples of positive integers with pairwise non-coprime components}}
\author{L\'aszl\'o T\'oth  \\
Department of Mathematics \\
University of P\'ecs \\
Ifj\'us\'ag \'utja 6, 7624 P\'ecs, Hungary \\
E-mail: {\tt ltoth@gamma.ttk.pte.hu}}
\date{}
\maketitle

\centerline{Journal of Integer Sequences {\bf 27} (2024), Article 24.8.5}

\begin{abstract} We use the convolution method for arithmetic functions of several variables to deduce 
an asymptotic formula for the number of $k$-tuples of positive integers with components which are pairwise non-coprime and $\le x$. More generally, we obtain asymptotic formulas on the number of $k$-tuples $(n_1,\ldots,n_k)\in \N^k$ such that at 
least $r$ pairs $(n_i,n_j)$, respectively exactly $r$ pairs are coprime. Our results answer the questions raised by Moree (2005, 2014), and generalize and refine related results obtained by Heyman (2014) and Hu (2014).
\end{abstract}

{\sl 2020 Mathematics Subject Classification}: 11A25, 11N25, 11N37, 05A15, 05C07

{\sl Key Words and Phrases}: pairwise coprime integers, pairwise non-coprime integers, 
asymptotic density, asymptotic formula, multiplicative function of several variables, 
inclusion-exclusion principle, simple graph, vertex cover

\section{Introduction and motivation}

Let $\N=\{1,2,\ldots \}$ and let $k\in \N$, $k\ge 2$. It is well-known that the asymptotic density of the $k$-tuples
$(n_1,\ldots,n_k)\in \N^k$ having relatively prime (coprime) components is $1/\zeta(k)$. This result goes back to the work of Ces\`{a}ro,  Dirichlet, Mertens and others. See, e.g., \cite{Chr1956,FerFer2021,Nym1972,Tot2001,Tot2002}. More exactly, one has the asymptotic estimate
\begin{equation} \label{rel_prime}
	\sum_{\substack{n_1,\ldots,n_k\le x\\ \gcd(n_1, \ldots,n_k)=1}} 1 
	= \frac{x^k}{\zeta(k)}+\begin{cases} O(x\log x), & \text{ if $k=2$};\\
		O(x^{k-1}),  & \text{ if $k\ge 3$}. 
	\end{cases}
\end{equation}

The asymptotic density of the $k$-tuples $(n_1,\ldots,n_k)\in \N^k$ with pairwise coprime components is 
\begin{equation} \label{density_pairwise_rel_prime}
	A_k =\prod_p \left(1-\frac1{p}\right)^{k-1} \left(1+\frac{k-1}{p}\right)= \prod_p \left(1+\sum_{j=2}^k (-1)^{j-1} (j-1) 
	\binom{k}{j} \frac1{p^j}\right),
\end{equation}
and we have the asymptotic formula  
\begin{equation} \label{form}
	\sum_{\substack{n_1,\ldots,n_k\le x\\ \gcd(n_i,n_j)=1\\ 1\le i<j\le k}} 1 
	= A_k x^k + O(x^{k-1}(\log x)^{k-1}),
\end{equation}
valid  for every fixed $k\ge 2$, proved by the author \cite{Tot2002} using an inductive process on $k$. The value of $A_k$ was also deduced by Cai and Bach \cite[Th.\ 3.3]{CaiBac2003} using probabilistic arguments. Formula \eqref{form} has been reproved by the author \cite{Tot2016}, in a more general form, namely by investigating 
$m$-wise relatively prime integers (that is, every $m$ of them are relatively prime) and by using the convolution method for functions of several variables. 
Note that the asymptotic formula for $m$-wise coprime integers was first proved by Hu \cite{Hu2013} by the inductive method with a weaker error term.

Now consider pairwise non-coprime positive integers $n_1,\ldots,n_k$, satisfying 
$\gcd(n_i,n_j)\ne 1$ for all $1\le i<j\le k$. Let $\beta=\beta_k: \N^k\to \{0,1\}$ denote the characteristic function
of $k$-tuples having this property, that is, 
\begin{equation} \label{def_beta}
	\beta(n_1,\ldots,n_k)= \begin{cases} 
		1, & \text{ if $n_1,\ldots,n_k$ are pairwise non-coprime}; \\
		0, & \text{ otherwise}.
	\end{cases}    
\end{equation}

Moree \cite{Mor2005,Mor2014} and Heyman \cite{Hey2014} raised the question to find the asymptotic 
density $C_k$ of $k$-tuples with pairwise non-coprime components. If $k=2$, then the answer is 
immediate by \eqref{rel_prime}: $C_2= 1-1/\zeta(2)$. Heyman \cite{Hey2014} obtained the value $C_3$
and deduced an asymptotic formula for the sum $\sum_{n_1,n_2,n_3\le x} \beta(n_1,n_2,n_3)$  
by using functions of one variable and the inclusion-exclusion principle. The method in 
\cite{Hey2014} cannot be applied if $k\ge 4$. Using the inductive approach of the author 
\cite{Tot2002} and the inclusion-exclusion principle, Hu \cite{Hu2014} gave a formula for the 
asymptotic density $C_k$ ($k\ge 3$), with an incomplete proof.

Moree \cite{Mor2005,Mor2014} also formulated as an open problem to compute the density of $k$-tuples  $(n_1,\ldots,n_k)\in \N^k$ such that at least $r$ pairs $(n_i,n_j)$, respectively exactly $r$ pairs are coprime. A correct answer to this problem, but with some incomplete arguments has been given by Hu \cite[Cor.\ 3]{Hu2014}. In fact, Hu \cite[Th.\ 1]{Hu2014} also deduced a related asymptotic formula with remainder term concerning certain arbitrary coprimality conditions. See Theorem \ref{Th_delta}. Arias de Reyna and Heyman \cite{ReyHey2015} used a different method, based on certain properties of arithmetic functions of one variable, and improved the error term by Hu \cite{Hu2014}. 

See Sections \ref{Sect_Prev} and \ref{Sect_Main} for some more details on the above results.

In this paper we use a different approach to study these questions. Applying the convolution method for functions of several variables we first reprove 
Theorem \ref{Th_delta}. To do this we need a careful study of the Dirichlet series of the corresponding characteristic function. See Theorem \ref{Th_Dirichlet_series}. Our result concerning the related asymptotic formula, with the same error term as obtained in \cite{ReyHey2015}, and with new representations of the constant $A_G$ is contained in Theorem \ref{Th_delta_new}. Then we deduce asymptotic formulas with remainder terms on the number of $k$-tuples such that at 
least $r$ pairs $(n_i,n_j)$, respectively exactly $r$ pairs are coprime. See Theorem \ref{Th_exact_least}. In particular, we obtain an asymptotic formula for the function $\beta=\beta_k$, for every $k\ge 2$. See Corollary \ref{Cor_asymt_beta_k}. Our results generalize and refine those by Heyman \cite{Hey2014} and Hu \cite{Hu2014}.

Basic properties of arithmetic functions of $k$ variables are presented in Section \ref{subSect_Arithm_func}. 
Some lemmas related to the principle of inclusion-exclusion, used in the proofs are included in Section 
\ref{subSect_Fuctions}. The proofs of our main results are similar to those in \cite{Tot2016}, and 
are given in Section \ref{Sect_Proofs}. Some numerical examples are presented in Section \ref{Section_Ex}.

\section{Previous results} \label{Sect_Prev}

Heyman \cite{Hey2014} proved the asymptotic formula
\begin{align*}
	\sum_{n_1,n_2,n_3\le x} \beta(n_1,n_2,n_3) 
	= C_3 x^3 + O(x^2 (\log x)^2),
\end{align*}
where the constant $C_3$ is 
\begin{equation} \label{C_3}
	C_3= 1- 3 \prod_p \left(1-\frac1{p^2}\right) + 3 \prod_p \left(1-\frac{2}{p^2}+\frac1{p^3} \right)
	- \prod_p \left(1-\frac{3}{p^2}+\frac2{p^3}\right).
\end{equation}

Hu \cite{Hu2014} gave a formula for the asymptotic density $C_k$ of $k$-tuples with pairwise non-coprime components,
where $k\ge 3$. See \eqref{C_k_0}. It recovers \eqref{C_3} for $k=3$, and for $k=4$ it can be written as
\begin{align} \label{C_4}
	C_4  & = 1 - 6 \prod_p \left(1-\frac1{p^2}\right) + 3 \prod_p \left(1-\frac1{p^2}\right)^2 
	+ 12 \prod_p \left(1-\frac{2}{p^2}+\frac1{p^3} \right) \\ \nonumber 
	&  - 4 \prod_p \left(1-\frac{3}{p^2}+\frac{3}{p^3} -\frac1{p^4}\right)
	- 16 \prod_p \left(1-\frac{3}{p^2}+\frac{2}{p^3} \right) \\  \nonumber
	&  +  15 \prod_p \left(1-\frac{4}{p^2}+\frac{4}{p^3}-\frac1{p^4} \right) 
	- 6 \prod_p \left(1-\frac{5}{p^2}+\frac{6}{p^3} -\frac{2}{p^4} \right) \\ \nonumber
	&  + \prod_p \left(1-\frac{6}{p^2}+\frac{8}{p^3}-\frac3{p^4} \right).
\end{align}

Related to identity \eqref{C_4} we note that there are two typos in \cite{Hu2014}, namely $\prod_p (1-1/p)^2(1-2/p)$ on pages 7 and 8 should be $\prod_p (1-1/p)^2(1+2/p)$.

For a fixed $k\ge 2$ let $V=\{1,2,\ldots,k\}$, let $E$ be an arbitrary subset of the set $\{(i,j): 1\le i< j\le k \}$, 
and let take the coprimality conditions $\gcd(n_i,n_j)=1$ for $(i,j)\in E$. Following Hu \cite{Hu2014}, Arias de Reyna and Heyman 
\cite{ReyHey2015}, it is convenient and suggestive 
to consider the corresponding simple graph $G=(V,E)$, we call it coprimality graph, with set of vertices $V$ and set of edges $E$. 
Therefore, we will use the notation $E\subseteq V^{(2)}:= \{\{i,j\}: 1\le i<j\le k\}$, where the edges of $G$ are denoted 
by $\{i,j\}=\{j,i\}$, and will adopt some related graph terminology. 

Let $\delta_G$ denote the characteristic function attached to the graph $G$, defined by
\begin{equation} \label{def_delta_G}
	\delta_G(n_1,\ldots,n_k)= \begin{cases} 
		1, & \text{ if $\gcd(n_i,n_j)=1$ for every $\{i,j\}\in E$;} \\
		0, & \text{ otherwise},
	\end{cases}    
\end{equation}
and note that if $E=\emptyset$, that is, the graph $G$ has no edges, then $\delta_G(n_1,\ldots,n_k)=1$ for every $(n_1,\ldots,n_k)\in \N^k$.

Furthermore, let $i_m(G)$ be the number of independent sets $S$ of vertices in $G$ (i.e., no two 
vertices of $S$ are adjacent in $G$) of cardinality $m$. Also, for $F\subseteq E$ let $v(F)$ denote the number of 
distinct vertices appearing in $F$.

\begin{theorem} \label{Th_delta} Let $G=(V,E)$ be an arbitrary graph. With the above notation,
	\begin{equation} \label{form_E}
		\sum_{n_1,\ldots,n_k\le x} \delta_G(n_1,\ldots,n_k) 
		= A_G x^k + O(x^{k-1}(\log x)^{\vartheta_G}),
	\end{equation}
	where the constant $A_G$ is given by 
	\begin{align} \label{A_G}
		A_G & = \prod_p \left(\sum_{m=0}^k \frac{i_m(G)}{p^m} \left(1-\frac{1}{p}\right)^{k-m}\right)
		\\
		\label{A_G_version}
		& = \prod_p \left(\sum_{F\subseteq E} \frac{(-1)^{|F|}}{p^{v(F)}} \right),
	\end{align}
	and $\vartheta_G=d_G:=\max_{j\in V} \deg(j)$, denoting the maximum degree of the vertices of $G$.
\end{theorem}

Here $A_G$ is representing the asymptotic density of $k$-tuples $(n_1,\ldots,n_k)\in \N^k$ such that $\gcd(n_i,n_j)=1$ for 
$\{i,j\}\in E$. Theorem \ref{Th_delta} was first proved by Hu \cite{Hu2014} with the weaker exponent $\vartheta_G = k-1$ 
for every subset $E$ and with identity \eqref{A_G} for the constant $A_G$. Arias de Reyna and Heyman \cite{ReyHey2015} deduced Theorem \ref{Th_delta} 
by a different method, with the given exponent $\vartheta_G=d_G$ and identity \eqref{A_G_version}
for the constant $A_G$. 

Note that if we have the complete coprimality graph, namely if $E=V^{(2)}$, then $\delta_G$ is the characteristic function 
of the set of $k$-tuples with pairwise coprime components (see \eqref{vartheta}), and \eqref{form_E} recovers formula \eqref{form}.

\section{Preliminaries} \label{Sect_Prelim}

\subsection{Arithmetic functions of $k$ variables} 
\label{subSect_Arithm_func}

The Dirichlet convolution of the functions $f,g:\N^k\to \C$ is defined by
\begin{equation} \label{def_convo}
	(f*g)(n_1,\ldots,n_k)= \sum_{d_1\mid n_1, \ldots, d_k\mid n_k}
	f(d_1,\ldots,d_k) g(n_1/d_1, \ldots, n_k/d_k).
\end{equation}

Let $\mu=\mu_k:\N^k\to \{-1,0,1\}$ denote the M\"obius function of $k$ variables, defined as the inverse of the constant $1$ 
function under convolution \eqref{def_convo}. We have $\mu(n_1,\ldots,n_k)=\mu(n_1)\cdots \mu(n_k)$ for every $n_1,\ldots,n_k\in \N$, 
which recovers for $k=1$ the classical (one variable) M\"obius function. 

The Dirichlet series of a function $f:\N^k\to \C$ is given by
\begin{align*}
	D(f;s_1,\ldots,s_k):= \sum_{n_1,\ldots,n_k=1}^{\infty}
	\frac{f(n_1,\ldots,n_k)}{n_1^{s_1}\cdots n_k^{s_k}}.
\end{align*} 

If $D(f;s_1,\ldots,s_k)$ and $D(g;s_1,\ldots,s_k)$ are
absolutely convergent, where $s_1,\ldots,s_k\in \C$, then
$D(f*g;s_1,\ldots,s_k)$ is also absolutely convergent and
\begin{align*}
	D(f*g;s_1,\ldots,s_k) = D(f;s_1,\ldots,s_k) D(g;s_1,\ldots,s_k).
\end{align*}

We recall that a nonzero arithmetic function of $k$ variables $f:\N^k\to \C$ is said to be multiplicative if
\begin{align*}
	f(m_1n_1,\ldots,m_kn_k)= f(m_1,\ldots,m_k) f(n_1,\ldots,n_k)
\end{align*}
holds for all $m_1,n_1\ldots,m_k,n_k \in \N$ such that $\gcd(m_1\cdots m_k,n_1\cdots n_k)=1$.
If $f$ is multiplicative, then it is determined by the values
$f(p^{\nu_1},\ldots,p^{\nu_k})$, where $p$ is prime and
$\nu_1,\ldots,\nu_k\in \N \cup \{0\}$. More exactly, $f(1,\ldots,1)=1$ and
for all $n_1,\ldots,n_k\in \N$,
\begin{align*}
	f(n_1,\ldots,n_k)= \prod_p f(p^{\nu_p(n_1)}, \ldots,p^{\nu_p(n_k)}),
\end{align*}
where we use the notation $n=\prod_p p^{\nu_p(n)}$ for the prime power factorization of $n\in \N$, the product being over the
primes $p$ and all but a finite number of the exponents $\nu_p(n)$ are zero. 
Examples of multiplicative functions of $k$ variables are the GCD and LCM functions $\gcd(n_1,\ldots,n_k)$, $\lcm(n_1,\ldots,n_k)$
and the characteristic functions
\begin{align} \nonumber
	\varrho(n_1,\ldots,n_k) & = 
	\begin{cases}
		1, & \text{if $\gcd(n_1,\ldots,n_k)=1$;}\\
		0, & \text{otherwise,} 
	\end{cases} 
	\\ \label{vartheta}
	\vartheta(n_1,\ldots,n_k) & = 
	\begin{cases} 
		1, & \text{if $\gcd(n_i,n_j)=1$ for every $1\le i< j\le k$;}\\
		0, & \text{otherwise.}
	\end{cases}
\end{align}

If the function $f$ is multiplicative, then its Dirichlet series can be expanded into a (formal)
Euler product, that is,
\begin{equation} \label{Euler_product}
	D(f;s_1,\ldots,s_k)=  \prod_p \sum_{\nu_1,\ldots,\nu_k=0}^{\infty}
	\frac{f(p^{\nu_1},\ldots, p^{\nu_k})}{p^{\nu_1s_1+\cdots +\nu_k s_k}},
\end{equation}
the product being over the primes $p$. More exactly, for $f$ multiplicative, the series
$D(f;s_1,\ldots$, $s_k)$ with $s_1,\ldots,s_k\in \C$ is absolutely convergent if and only if
\begin{align*}
	\sum_p \sum_{\substack{\nu_1,\ldots,\nu_k=0\\ \nu_1+\cdots +\nu_k \ge 1}}^{\infty}
	\frac{|f(p^{\nu_1},\ldots, p^{\nu_k})|}{p^{\nu_1 \Re s_1+\cdots +\nu_k
			\Re s_k}} < \infty
\end{align*}
and in this case equality \eqref{Euler_product} holds. 

The mean value of a function $f:\N^k \rightarrow \C$ is
\begin{align*}
	M(f): = \lim_{x_1, \ldots, x_k\to \infty} \frac1{x_1\cdots x_k}
	\sum_{n_1\le x_1,\ldots, n_k\le x_k} f(n_1,\ldots,n_k),
\end{align*}
provided that this limit exists. As a generalization of Wintner's theorem (valid in the one variable
case), Ushiroya \cite[Th.\ 1]{Ush2012} proved the next result.

\begin{theorem} \label{Th_Wintner_gen}
	If $f$ is a function of $k$ variables,
	not necessary multiplicative, such that
	\begin{align*}
		\sum_{n_1,\ldots,n_k=1}^{\infty} \frac{|(\mu*f)(n_1,\ldots,n_k)|}{n_1\cdots n_k} < \infty,
	\end{align*}
	then the mean value\index{mean value} $M(f)$ exists, and
	\begin{align*}
		M(f)= \sum_{n_1,\ldots,n_k=1}^{\infty} \frac{(\mu*f)(n_1,\ldots,n_k)}{n_1\cdots n_k}.
	\end{align*}
\end{theorem}

For multiplicative functions the above result can be formulated as follows. See \cite[Prop.\ 19]{Tot2014}, \cite[Th.\ 4]{Ush2012}. 

\begin{theorem}
	Let $f:\N^k\to \C$ be a multiplicative function. Assume that
	\begin{align*}
		\sum_p \sum_{\substack{\nu_1,\ldots,\nu_k=0\\ \nu_1+\cdots +\nu_k \ge 1}}^{\infty}
		\frac{|(\mu*f)(p^{\nu_1},\ldots,p^{\nu_k})|}{p^{\nu_1+\cdots +\nu_k}} < \infty.
	\end{align*}
	
	Then the mean value $M(f)$ exists, and
	\begin{align*}
		M(f)= \prod_p \left(1-\frac1{p}\right)^k \sum_{\nu_1,\ldots,\nu_k=0}^{\infty}
		\frac{f(p^{\nu_1},\ldots,p^{\nu_k})}{p^{\nu_1+\cdots +\nu_k}}.
	\end{align*} 
\end{theorem}

See, e.g., Delange \cite{Del1969} and the survey by the author \cite{Tot2014} for these and 
some other related results on arithmetic functions of several variables. If $k=1$, i.e., in the case of functions 
of a single variable we recover some familiar properties.

\subsection{The functions $\delta_G$ and $\beta$}
\label{subSect_Fuctions}

Consider the function $\delta_G$ defined by \eqref{def_delta_G}.

\begin{lemma} \label{Lemma_multipl}  For every subset $E$, the function $\delta_G$  is multiplicative.
\end{lemma}

\begin{proof}
	This is a consequence of the fact that the gcd function $\gcd(m,n)$ is multiplicative, viewed as a function
	of two variables.
	To give a direct proof, let $m_1,n_1,\ldots,m_k,n_k\in \N$ such that $\gcd(m_1\cdots m_k, n_1\cdots n_k)=1$.
	Then we have
	\begin{align*}
		\delta_G(m_1n_1,\ldots,m_kn_k) & =  \begin{cases} 
			1, & \text{ if $\gcd(m_in_i,m_jn_j)=1$ for all $\{i,j\}\in E$;} \\
			0, & \text{ otherwise;}
		\end{cases}   \\
		& = \begin{cases}  1, & \text{ if $\gcd(m_i,m_j)\gcd(n_i,n_j)=1$ for all $\{i,j\}\in E$;} \\
			0, & \text{ otherwise;}
		\end{cases}    \\
		& = \begin{cases}  1, & \text{ if $\gcd(m_i,m_j)=1$ for all $\{i,j\}\in E$;} \\
			0, & \text{ otherwise;}
		\end{cases} \\    
		& \quad \times
		\begin{cases}  1, & \text{ if $\gcd(n_i,n_j)=1$ for all $\{i,j\}\in E$;} \\
			0, & \text{ otherwise;}
		\end{cases} \\
		& = \delta_G(m_1,\ldots, m_k) \delta_G(n_1,\ldots,n_k),
	\end{align*}
	finishing the proof.
\end{proof}

The function $\beta$ given by \eqref{def_beta} is not multiplicative. However, by the inclusion-exclusion principle
it can be written as an alternating sum of certain multiplicative functions $\delta_G$.

More generally, for $r\ge 0$ we define the functions $\beta_r=\beta_{k,r}$ and $\beta'_r=\beta'_{k,r}$ by
\begin{align} \label{def_beta_exact}
	\beta_r(n_1,\ldots,n_k) & = \begin{cases} 
		1, & \text{ if exactly $r$ pairs $(n_i,n_j)$ with $1\le i<j\le k$ are coprime}; \\
		0, & \text{ otherwise},
	\end{cases} \\
	\label{def_beta_least}
	\beta'_r(n_1,\ldots,n_k) & = \begin{cases} 
		1, & \text{ if at least $r$ pairs $(n_i,n_j)$ with $1\le i<j\le k$ are coprime}; \\
		0, & \text{ otherwise}.
	\end{cases}    
\end{align}

If $r=0$, then $\beta_0=\beta$.

\begin{lemma} Let $k\ge 2$ and $r\ge 0$. For every $n_1,\ldots,n_k\in \N$,
	\begin{align} \label{beta_r}
		\beta_r(n_1,\ldots,n_k) & = \sum_{j=r}^{k(k-1)/2} (-1)^{j-r} \binom{j}{r} \sum_{\substack{E\subseteq V^{(2)} \\ |E|=j}} \delta_G(n_1,\ldots,n_k),
		\\ 
		\label{beta}
		\beta(n_1,\ldots,n_k) & = \sum_{j=0}^{k(k-1)/2} (-1)^j \sum_{\substack{E\subseteq V^{(2)} \\ |E|=j}} 
		\delta_G(n_1,\ldots,n_k).
	\end{align}
\end{lemma}

\begin{proof} Given $n_1,\ldots,n_k\in \N$, assume that for $1\le i<j\le k$ condition $\gcd(n_i,n_j)=1$ holds for $t$ times, where $0\le t \le k(k-1)/2$. Then the right hand side of \eqref{beta_r} is
	\begin{align*} 
		N_r:=\sum_{j=r}^t (-1)^{j-r} \binom{j}{r} \binom{t}{j}.
	\end{align*}
	
	If $t<r$, then this is the empty sum, and $N_r=0=\beta_r(n_1,\ldots,n_k)$.
	If $t\ge r$, then
	\begin{align*} 
		N_r = \binom{t}{r} \sum_{j=r}^t (-1)^{j-r} \binom{t-r}{j-r} = \binom{t}{r} \sum_{m=0}^{t-r} (-1)^m \binom{t-r}{m} 
		= \begin{cases} 1, & \text{ if $t=r$};\\ 0, & \text{ if $t>r$}, 
		\end{cases} 
	\end{align*}
	which is exactly $\beta_r(n_1,\ldots,n_k)$.
	
	In the case $r=0$ we obtain identity \eqref{beta}.
\end{proof}

\begin{lemma} Let $k\ge 2$ and $r\ge 1$. For every $n_1,\ldots,n_k\in \N$,
	\begin{equation} \label{beta_least_r}
		\beta'_r(n_1,\ldots,n_k)= \sum_{j=r}^{k(k-1)/2} (-1)^{j-r} \binom{j-1}{r-1} \sum_{\substack{E\subseteq V^{(2)} \\ |E|=j}} \delta_G(n_1,\ldots,n_k).
	\end{equation}
\end{lemma}

\begin{proof} We have by using \eqref{beta_r}, 
	\begin{align*} 
		\beta'_r(n_1,\ldots,n_k) & = \sum_{t=r}^{k(k-1)/2} \beta_t(n_1,\ldots,n_k) \\
		& = \sum_{t=r}^{k(k-1)/2} \, \sum_{j=t}^{k(k-1)/2} (-1)^{j-t} \binom{j}{t} \sum_{\substack{E\subseteq V^{(2)} \\ |E|=j}} \delta_G(n_1,\ldots,n_k) \\
		& = \sum_{j=r}^{k(k-1/2} (-1)^j \sum_{\substack{E\subseteq V^{(2)} \\ |E|=j}} \delta_G(n_1,\ldots,n_k) \sum_{t=r}^j (-1)^t \binom{j}{t}, 
	\end{align*}
	where the inner sum is $(-1)^r\binom{j-1}{r-1}$, finishing the proof.
\end{proof}

The above identities are similar to some known generalizations of the principle of 
in\-clu\-sion-exclusion. See, e.g., the books by Aigner \cite[Sect.\ 5.1]{Aig2007} and Stanley \cite[Ch.\ 2]{Sta1997}. 

\section{Main results} \label{Sect_Main}

Given a graph $G=(V,E)$, the asymptotic density $A_G$ of the of $k$-tuples $(n_1,\ldots,n_k)\in \N^k$ such that $\gcd(n_i,n_j)=1$ for $\{i,j\}\in E$ is the mean value of the characteristic function $\delta_G$ defined by \eqref{def_delta_G}. According to Theorem \ref{Th_Wintner_gen}, $A_G=D(\mu*\delta_G,1,\ldots,1)$, provided that 
this series is absolutely convergent. We show this by a careful study of the Dirichlet series of the function 
$\delta_G$.

To formulate our results we need the following additional notation. For a graph $G=(V,E)$ let $I$ be the set of non-isolated vertices of $G$, and 
$J$ be a (minimum) vertex cover of $G$, that is, a set of vertices that includes at least one endpoint of every edge (of smallest possible size).
The notation $\sum_{L\subseteq J}^{'}$ will mean the sum over independent subsets $L$ of $J$ (no two 
vertices of $L$ are adjacent in $G$). Also, let $N(j)$ denote the neighbourhood of a vertex $j$, and for a subset $L$ of $V$ let $N(L)= \cup_{j\in L} N(j)$.

\begin{theorem} \label{Th_Dirichlet_series} Let $k\ge 2$ and $G=(V,E)$ be an arbitrary graph. 
	Then, with the above notation,
	\begin{align*}
		\sum_{n_1,\ldots,n_k=1}^{\infty} \frac{\delta_G(n_1,\ldots,n_k)}{n_1^{s_1}\cdots n_k^{s_k}}
		=\zeta(s_1)\cdots \zeta(s_k) D'_G(s_1,\ldots,s_k), 
	\end{align*}
	where 
	\begin{align} \label{delta_Euler_prod}
		D'_G(s_1,\ldots,s_k) & =  \prod_p \Bigg( \sideset{}{'} \sum_{L\subseteq J} \prod_{\ell \in L} \frac1{p^{s_{\ell}}}  \prod_{i\in (J\setminus L) \cup 
			(N(L)\setminus J)} \left(1-\frac1{p^{s_i}}\right)\Bigg) \\
		\nonumber
		& =  \prod_p \Bigg(1 - \sum_{\{i,j\}\in E} \frac1{p^{s_i+s_j}} + \sum_{j=3}^{|I|} \sum_{\substack{i_1,\ldots,i_j\in I\\   i_1<\cdots <i_j}} 
		\frac{c(i_1,\ldots,i_j)}{p^{s_{i_1}+\cdots +s_{i_j}}}\Bigg),
	\end{align} 
	where $c(i_1,\ldots,i_j)$ are some integers, depending on $i_1,\ldots,i_j$, but not on $p$.
	
	Furthermore, $D'(s_1,\ldots,s_k)$ with $s_1,\ldots,s_k\in \C$ is absolutely convergent provided that $\Re (s_{i_1}+\cdots +s_{i_j})>1$ for every $i_1,\ldots,i_j\in I$ with $i_1< \cdots < i_j$, $2\le j\le |I|$.
\end{theorem}

\begin{remark}
	{\rm By choosing $J=I$ or $J=\{1,\ldots,k\}$ the sum over $L$ in identity \eqref{delta_Euler_prod} has more terms
		than in the case of a minimum vertex cover $J$ of $G$.
		However, if $J=I$, then $N(L)\setminus I=\emptyset$ for every $L$, and \eqref{delta_Euler_prod} takes the slightly simpler form 
		\begin{align*}
			D'_G(s_1,\ldots,s_k) =  \prod_p \Bigg( \sideset{}{'} \sum_{L\subseteq I} \prod_{\ell \in L} \frac1{p^{s_{\ell}}}  \prod_{i\in I\setminus L} \left(1-\frac1{p^{s_i}}\right)\Bigg),
		\end{align*}
		and similarly if $J=\{1,\ldots,k\}$.
	}
\end{remark}

Next we prove by the convolution method the asymptotic formula 
already given in Theorem \ref{Th_delta}. This approach leads to new representations of the constant $A_G$. 

\begin{theorem} \label{Th_delta_new} Asymptotic formula \eqref{form_E} holds with the exponent $\vartheta_G=d_G$ in the error term, and with the constant 
	\begin{align} \nonumber 
		A_G & = \sum_{n_1,\ldots,n_k=1}^{\infty} \frac{(\mu*\delta_G)(n_1,\ldots,n_k)}{n_1\cdots n_k} \\
		\nonumber
		& = \prod_p \left(1-\frac1{p}\right)^k \sum_{\nu_1,\ldots,\nu_k=0}^{\infty} \frac{\delta_G(p^{\nu_1},\ldots,p^{\nu_k})}{p^{\nu_1+\cdots + \nu_k}} \\
		\label{A_G_version_new}
		& = \prod_p \Bigg( \sideset{}{'} \sum_{L\subseteq J} \frac1{p^{|L|}} \left(1-\frac1{p}\right)^{|J\setminus L|+ 
			|N(L) \setminus J|} \Bigg). 
	\end{align}
\end{theorem}

\begin{remark}
	{\rm If $J=I$, then $N(L)\setminus I=\emptyset$ for every $L$, and  \eqref{A_G_version_new}  gives
		\begin{align*}
			A_G & =  \prod_p \Bigg( \sideset{}{'} \sum_{L\subseteq I} \frac1{p^{|L|}} \left(1-\frac1{p}\right)^{|I\setminus L|} \Bigg) \\
			& = \prod_p \Bigg( \sum_{m=0}^{|I|} \frac{i_m(G,I)}{p^m}\left(1-\frac1{p}\right)^{|I|-m}\Bigg),
		\end{align*}
		where $i_m(G,I)$ denotes the number of independent subsets of $I$ of cardinality $m$ in the graph $G$.
		Similarly, by choosing $J=\{1,\ldots,k\}$, \eqref{A_G_version_new} reduces to identity \eqref{A_G} by Hu \cite{Hu2014}.
	}
\end{remark}

Now consider the functions $\beta_r$ and $\beta'_r$ defined by \eqref{def_beta_exact} and \eqref{def_beta_least}.

\begin{theorem} \label{Th_exact_least} Let $k\ge 2$.
	Then for $r\ge 0$ we have 
	\begin{equation} \label{asympt_form_exact} 
		\sum_{n_1,\ldots,n_k\le x} \beta_r(n_1,\ldots,n_k) = 
		C_r x^k  + O(x^{k-1} (\log x)^{k-1}),
	\end{equation}
	and for $r\ge 1$,
	\begin{equation} \label{asympt_form_least} 
		\sum_{n_1,\ldots,n_k\le x} \beta'_r(n_1,\ldots,n_k) = 
		C'_r x^k  + O(x^{k-1} (\log x)^{k-1}),
	\end{equation}
	where
	\begin{equation}  \label{C_k_r}
		C_r = C_{k,r} = \sum_{j=r}^{k(k-1)/2} (-1)^{j-r} \binom{j}{r}\sum_{\substack{E\subseteq V^{(2)}\\ |E|=j}} A_G, 
	\end{equation}
	and 
	\begin{equation} \label{C'_k_r}
		C'_r = C'_{k,r} = \sum_{j=r}^{k(k-1)/2} (-1)^{j-r} \binom{j-1}{r-1}\sum_{\substack{E\subseteq V^{(2)}\\ |E|=j}} A_G 
	\end{equation}
	are the asymptotic densities of the $k$-tuples $(n_1,\ldots,n_k)\in \N^k$ such that $\gcd(n_i,n_j)=1$ occurs exactly $r$ times,
	respectively at least $r$ times, with $A_G$ given in Theorems \ref{Th_delta} 
	and \ref{Th_delta_new}.
\end{theorem}

We remark that identities \eqref{C_k_r} and \eqref{C'_k_r} have been obtained by Hu \cite[Cor.\ 3]{Hu2014} with an incomplete proof.

\begin{corollary} \label{Cor_asymt_beta_k} \textup{($r=0$)} We have
	\begin{align*}
		\sum_{n_1,\ldots, n_k \le x} \beta(n_1,\ldots, n_k) 
		= C_k x^k + O(x^{k-1} (\log x)^{k-1}),
	\end{align*}
	where
	\begin{align} \nonumber
		C_k= C_{k,0} & = \sum_{n_1,\ldots,n_k=1}^{\infty} \frac{(\mu*\beta)(n_1,\ldots,n_k)}{n_1\cdots n_k} \\
		\label{C_k_0}
		& = \sum_{j=0}^{k(k-1)/2} (-1)^j \sum_{\substack{E\subseteq V^{(2)}\\ |E|=j}} A_G.
	\end{align}
\end{corollary}

Identity \eqref{C_k_0} has been obtained by Hu \cite[Cor.\ 3]{Hu2014}.

Note that if $G$ and $G'$ are isomorphic graphs then the corresponding densities $A_G$ and $A_{G'}$ are equal.
The asymptotic densities $C_{k,r}$, $C'_{k,r}$ and $C_k$ can be computed for given values of $k$ and $r$ 
from identities \eqref{C_k_r}, \eqref{C'_k_r} and \eqref{C_k_0}, respectively by determining the cardinalities
of the isomorphism classes of graphs $G$ with $k$ vertices and $j$ edges ($0\le j\le k(k-1)/2$) and by computing the 
corresponding values of $A_G$. In particular, $C_3$ and $C_4$ given by \eqref{C_3} and \eqref{C_4} can be obtained 
in this way.

\section{Proofs} \label{Sect_Proofs}

We first prove the key result of our treatment.

\begin{proof}[Proof of Theorem {\rm \ref{Th_Dirichlet_series}}] 
	We have 
	\begin{align*}
		D(\delta_G;s_1,\ldots,s_k) = \sum_{n_1,\ldots,n_k=1}^{\infty} \frac{\delta_G(n_1,\ldots,n_k)}{n_1^{s_1}\cdots n_k^{s_k}}
		= \sum_{\substack{n_1,\ldots,n_k=1\\ \gcd(n_{i_1},n_{i_2})=1\\ \{i_1,i_2\} \in E}}^{\infty}  \frac1{n_1^{s_1}\cdots n_k^{s_k}}.
	\end{align*}
	
	Let $I$ denote the set of non-isolated vertices of $G$. Then 
	\begin{align*}
		D(\delta_G;s_1,\ldots,s_k) & = \prod_{i \notin I} \zeta(s_i) \sum_{\substack{n_i \ge 1, \, i \in I \\ \gcd(n_{i_1},n_{i_2})=1, \, \{i_1,i_2\} \in E}} \prod_{i \in I} \frac1{n_i^{s_i}} \\
		& = \prod_{i \notin I} \zeta(s_i) \prod_p \Bigg( \sum_{\substack{\nu_i\ge 0, \, i \in I\\ \nu_{i_1}\nu_{i_2}=0, \, \{i_1,i_2\}\in E}} \ \frac1{p^{\sum_{i \in I} \nu_i s_i}} \Bigg) \\
		& =: \prod_{i \notin I} \zeta(s_i) \prod_p S_p,
	\end{align*}
	say, using that the function $\delta_G$ is multiplicative by Lemma \ref{Lemma_multipl}.
	
	Now choose a (minimum) vertex cover $J$. Then $\nu_j$ ($j\in J$) cover all the conditions $\nu_{i_1}\nu_{i_2}=0$ with $\{i_1,i_2\}\in E$, that is, for every $\{i_1,i_2\}\in E$ there is $j\in J$ such that $j=i_1$ or $j=i_2$.
	Group the terms of the sum $S_p$ according to the subsets $L=\{\ell \in J: \nu_{\ell}\ge 1 \}$ of $J$. 
	Here $\nu_j=0$ for every $j \in J\setminus L$. Note that $L$ cannot contain any two adjacent vertices. Also, for such a fixed subset $L \subseteq J$ let $M$ be the set of indexes $m$ such that $\nu_m$ is forced to be zero by $L$. 
	More exactly, let $M= \{m\in I\setminus J: \text{there is $\ell \in L$ with $\{m,\ell\}\in E$} \}$. If $m\in M$, then 
	$\nu_m\nu_{\ell}=0$ for some $\ell \in L$. Since $\nu_{\ell}\ge 1$, we obtain $\nu_m=0$. Here $M=N(L)\setminus J$, where
	$N(L)$ is set of vertices adjacent to vertices in $L$.
	
	Let $\sum_{L\subseteq J}^{'}$ denote the sum over subsets $L$ of $J$ that have no adjacent vertices.
	We obtain
	\begin{align*}
		S_p  & =\sideset{}{'} \sum_{L\subseteq J} \sum_{\substack{\nu_{\ell} \ge 1, \, \ell \in L\\ \nu_j=0, \, j\in J\setminus L\\ 
				\nu_m=0, \, m\in M \\ \nu_i\ge 0, \, i\in I\setminus (J\cup M)}} \frac1{p^{\sum_{i\in I} \nu_i s_i}} \\
		& = \sideset{}{'} \sum_{L\subseteq J} \sum_{\nu_{\ell} \ge 1, \, \ell \in L} \frac1{p^{\sum_{\ell \in L} \nu_{\ell} s_{\ell}}} 
		\sum_{\nu_i\ge 0, \, i\in I\setminus (J\cup M)} \frac1{p^{\sum_{i \in I\setminus (J\cup M)} \nu_i s_i}} \\
		& = \sideset{}{'} \sum_{L\subseteq J} \prod_{\ell \in L} \frac1{p^{s_{\ell}}} \left(1-\frac1{p^{s_{\ell}}}\right)^{-1} \prod_{i \in I\setminus (J\cup M)} \left(1-\frac1{p^{s_i}} \right)^{-1} \\
		& = \prod_{i\in I} \left(1-\frac1{p^{s_i}}\right)^{-1} \sideset{}{'} \sum_{L\subseteq J} \prod_{\ell \in L} \frac1{p^{s_{\ell}}}  
		\prod_{i\in (J\setminus L) \cup M} \left(1-\frac1{p^{s_i}}\right)
		\\
		& =: \prod_{i\in I} \left(1-\frac1{p^{s_i}}\right)^{-1} T_p,  
	\end{align*}
	say. We deduce that 
	\begin{align*}
		\prod_p S_p = \prod_{i\in I} \zeta(s_i) \prod_p  T_p ,
	\end{align*}
	which shows that 
	\begin{align*}
		D(\delta_G;s_1,\ldots,s_k) = \Big(\prod_{i=1}^k \zeta(s_i) \Big) D'(s_1,\ldots,s_k),
	\end{align*}
	where
	\begin{equation} \label{mu_delta_G_product}
		D'(s_1,\ldots,s_k)= \prod_p T_p = \prod_p \Bigg(\sideset{}{'} \sum_{L\subseteq J} \prod_{\ell \in L} \frac1{p^{s_{\ell}}}  \prod_{i\in (J\setminus L) \cup M} \left(1-\frac1{p^{s_i}}\right)\Bigg). 
	\end{equation}
	
	Let us investigate the terms of the sum $\sum_{L\subseteq J}^{'}$ in \eqref{mu_delta_G_product}. If $L=\emptyset$, that 
	is $\nu_j=0$ for every $j\in J$, then $M=\emptyset$ and we have 
	\begin{equation} \label{L_0}
		\prod_{i\in J} \left(1-\frac1{p^{s_i}}\right)= 1-\sum_{i\in J} \frac1{p^{s_i}} + \sum_{i,j\in J, \, i<j} \frac1{p^{s_i+s_j}}-\cdots .
	\end{equation}
	
	If $L=\{i_0\}$ for some fixed $i_0\in J$, then obtain, with $M = M_{i_0}: = N(i_0)\setminus J$, 
	\begin{align*} 
		\frac1{p^{s_{i_0}}} \prod_{\substack{i\in J\cup M_{i_0}\\ i\ne i_0}} \left(1-\frac1{p^{s_i}}\right) = \frac1{p^{s_{i_0}}} -\sum_{\substack{i\in J\cup M_{i_0}\\ i\ne i_0}} \frac1{p^{s_{i_0}+s_i}} + \sum_{\substack{i,j\in J\cup M_{i_0} \\ i_0\ne i<j\ne i_0}} \frac1{p^{s_{i_0}+s_i+s_j}}-\cdots .
	\end{align*}
	
	Here if $i_0=t$ runs over $J$, then we have the terms 
	\begin{equation} \label{L_1}
		\sum_{t\in J} \frac1{p^{s_t}} -\sum_{\substack{i\in J\cup M_t\\ t\in J\\  i\ne t}} \frac1{p^{s_t + s_i}} + 
		\sum_{\substack{i,j\in J\cup M_t \\ t\in J, i\ne t, j\ne t\\ i<j}} \frac1{p^{s_t+s_i+s_j}}- \cdots .
	\end{equation}
	
	If $L=\{i_0,i'_0\}$ with some fixed $i_0, i'_0\in J$, $i_0<i'_0$, which are not adjacent, then obtain, 
	with $M = M_{i_0,i'_0}:= N(\{i_0,i'_0\})\setminus J$, 
	\begin{equation} \label{L_i_0}
		\frac1{p^{s_{i_0}+s_{i'_0}}} \prod_{\substack{i\in J\cup M_{i_0,i'_0} \\ i\ne i_0,i'_0}} \left(1-\frac1{p^{s_i}}\right) = \frac1{p^{s_{i_0}+s_{i'_0}}} -\sum_{\substack{i\in J\cup M_{i_0,i'_{0}}\\ i\ne i_0,i'_0}} \frac1{p^{s_{i_0}+s_{i'_0}+ s_i}} + \cdots  .
	\end{equation}
	
	If $i_0=t, i'_0=v$ run over $J$, then we obtain from \eqref{L_i_0},
	\begin{equation} \label{L_2} 
		\sum_{\substack{t,v\in J\\ t<v\\ t,v \text{ not adjacent}}}  \frac1{p^{s_t+s_v}} -\sum_{\substack{i\in J\cup M_{t,v}\\ t,v\in J \text{ not adjacent}\\  i\ne t,v}} \frac1{p^{s_t + s_v+ s_i}} + \cdots .
	\end{equation}
	
	Putting together \eqref{L_0}, \eqref{L_1} and \eqref{L_2} we obtain the sum $S$, where
	\begin{align*} 
		S & = 1- \sum_{i\in J} \frac1{p^{s_i}} + \sum_{i,j\in J, \, i<j} \frac1{p^{s_i+s_j}} +
		\sum_{t\in J} \frac1{p^{s_t}} -\sum_{\substack{i\in J\cup M_t\\ t\in J\\  i\ne t}} \frac1{p^{s_t + s_i}} 
		\\ & \quad +\sum_{\substack{t,v\in J\\ t<v\\ t,v \text{ not adjacent}}}  \frac1{p^{s_t+s_v}} \pm \text{ other terms} \\
		& = 1- \sum_{\substack{i,t\in J\\ i,t \text{ adjacent}}} \frac1{p^{s_t + s_i}} \pm \text{ other terms},
	\end{align*}
	where the terms $\pm 1/p^i$ with $i\in J$ cancel out. Also the terms $\pm 1/p^{s_i+s_j}$ with $i,j\in J$ (each appearing twice) cancel out, excepting when $i,j$ are adjacent. Here for the ``other terms'', including the terms obtained if $L$ has at least three elements, the exponents of $p$ are sums of at least three distinct values $s_i$, $s_j$, $s_{\ell}$ with $i,j,\ell \in I$. 
	
	Hence the infinite product \eqref{mu_delta_G_product} is absolutely convergent 
	provided the given condition. 
\end{proof} 

\begin{remark} \label{Remark_delta}
	{\rm It turns out that the function $\mu*\delta_G$ is multiplicative (in general not symmetric in the variables) and for all prime powers $p^{\nu_1},\ldots,p^{\nu_k}$,
		\begin{equation} \label{mu_delta_primes}
			(\mu*\delta_G)(p^{\nu_1},\ldots,p^{\nu_k})=
			\begin{cases} 1, & \text{ if $\nu_1=\cdots=\nu_k=0$};\\
				c(\nu_1,\ldots,\nu_k), & \text{ if $\nu_1,\ldots,\nu_k\in \{0,1\}, \ j:=\nu_1+\cdots+\nu_k \ge 2$}; \\
				0, & \text{ otherwise},
			\end{cases}
		\end{equation} 
		where $c(\nu_1,\ldots,\nu_k)$ are some integers, depending on $\nu_1,\ldots,\nu_k$, 
		but not on $p$.
		
		Note that $(\mu*\delta_G)(p^{\nu_1},\ldots,p^{\nu_k})=0$ provided that
		$\nu_i\ge 2$ for at least one $1\le i\le k$, or
		$\nu_1,\ldots,\nu_k\in \{0,1\}$ and $\nu_1+\cdots+\nu_k=1$. If 
		$\nu_1,\ldots,\nu_k\in \{0,1\}$ and $\nu_1+\cdots+\nu_k=2$, say $\nu_{i_0}=\nu_{i'_0}=1$
		and $\nu_i=0$ for $i\ne i_0, i'_0$, then $(\mu*\delta_G)(p^{\nu_1},\ldots,p^{\nu_k})= -1$
		if $i_0$ and $i'_0$ are adjacent in the graph $G$ and $0$ otherwise. 
	}
\end{remark}

\begin{proof}[Proof of Theorem {\rm \ref{Th_delta_new}}]
	Write
	\begin{align} \nonumber
		\sum_{n_1,\ldots,n_k\le x} \delta_G(n_1,\ldots,n_k) & = \sum_{n_1,\ldots,n_k\le x} \sum_{d_1\mid n_1,\ldots, d_k\mid n_k}
		(\mu* \delta_G)(d_1,\ldots,d_k) \\
		\nonumber
		& = \sum_{d_1,\ldots,d_k\le x} (\mu * \delta_G)(d_1,\ldots,d_k) \left\lfloor \frac{x}{d_1} \right\rfloor  
		\cdots \left\lfloor \frac{x}{d_k} \right\rfloor \\
		\nonumber 
		& = \sum_{d_1,\ldots,d_k\le x} (\mu * \delta_G)(d_1,\ldots,d_k) \left(\frac{x}{d_1} +O(1) \right) \cdots \left(\frac{x}{d_k} +O(1) \right) \\
		\label{main_term}
		& = x^k \sum_{d_1, \ldots, d_k\le x}
		\frac{(\mu*\delta_G)(d_1,\ldots,d_k)}{d_1\cdots d_k} + R_k(x),
	\end{align}
	with
	\begin{align*}
		R_k(x)\ll \sum_{u_1,\ldots,u_r} x^{u_1+\cdots+u_k}
		\sum_{d_1,\ldots, d_k\le x} \frac{|(\mu*\delta_G)(d_1,\ldots,d_k)|}{d_1^{u_1}\cdots d_k^{u_k}},
	\end{align*}
	where the first sum is over $u_1,\ldots,u_k\in \{0, 1\}$ such that
	at least one $u_i$ is $0$. Let $u_1,\ldots,u_k$ be fixed and assume
	that $u_{i_0}=0$. Since $(x/d_i)^{u_i}\le x/d_i$ for every $i$ ($1\le i\le k$) we have
	\begin{align} \nonumber
		A & := x^{u_1+\cdots+u_k} \sum_{d_1,\ldots, d_k\le x}
		\frac{|(\mu*\delta_G)(d_1,\ldots,d_k)|}{d_1^{u_1}\cdots d_k^{u_k}} \\
		\nonumber
		& \le x^{k-1} \sum_{d_1,\ldots,d_k\le x} \frac{|(\mu*\delta_G)(d_1,\ldots,d_k)|}{\prod_{1\le i\le k, i\ne i_0}
			d_i} \\ \nonumber
		& \le x^{k-1} \prod_{p\le x} \sum_{\nu_1,\ldots, \nu_k=0}^{\infty}
		\frac{|(\mu*\delta_G)(p^{\nu_1},\ldots,p^{\nu_k})|}{p^{\sum_{1\le i\le k, i\ne i_0} \nu_i}}
		\\ \label{A}
		& = x^{k-1} \prod_{p\le x} \left(1+\frac{c_{i_0,1}}{p} +\frac{c_{i_0,2}}{p^2}+\cdots+ 
		\frac{c_{i_0,k-1}}{p^{k-1}} \right), 
	\end{align}
	cf. \eqref{mu_delta_primes}, where $c_{i_0,j}$ ($1\le j\le k-1$) are certain non-negative integers. Here $c_{i_0,1}=\deg(i_0)$, 
	the degree of $i_0$, according to Remark \ref{Remark_delta}.  We obtain that 
	\begin{align*}
		A\ll x^{k-1} \prod_{p\le x} \left(1+\frac1{p}\right)^{\deg(i_0)} \ll
		x^{k-1} (\log x)^{\deg(i_0)}
	\end{align*}
	by Mertens' theorem. This shows that
	\begin{equation}  \label{Mertens}
		R_k(x)\ll x^{k-1}(\log x)^{\max \deg(i_0)}.
	\end{equation} 
	
	Furthermore, for the main term of \eqref{main_term} we have
	\begin{align*}
		\sum_{d_1, \ldots, d_k\le x} \frac{(\mu*\delta_G)(d_1,\ldots,d_k)}{d_1\cdots d_k}
	\end{align*}
	\begin{equation} \label{roro}
		= \sum_{d_1,\ldots,d_k=1}^{\infty}
		\frac{(\mu*\delta_G)(d_1,\ldots,d_k)}{d_1 \cdots d_k} - \sum_{\emptyset \ne I \subseteq \{1,\ldots,k\}}
		\sum_{\substack{d_i>x, \, i\in I\\ d_j\le x,
				\, j\notin I}} \frac{(\mu*\delta_G)(d_1,\ldots,d_k)}{d_1\cdots d_k},
	\end{equation}
	where the series is convergent by Theorem \ref{Th_Dirichlet_series} and its
	sum is $D(\mu*\delta_G;1,\ldots,1)=A_G$.
	
	Let $I$ be fixed with $|I|=t$. We estimate the sum
	\begin{align*}
		B:= \sum_{\substack{d_i>x, \, i\in I\\ d_j\le x,
				\, j\notin I}} \frac{|(\mu*\delta_G)(d_1,\ldots,d_k)|}{d_1\cdots d_k}.
	\end{align*}
	
	Case I. Assume that $|I|=t\ge 3$. If $0<\varepsilon <1/2$, then
	\begin{align*}
		B & = \sum_{\substack{d_i>x, \, i\in I\\ d_j\le x,
				\, j\notin I}} \frac{|(\mu*\delta_G)(d_1,\ldots,d_k)| \prod_{i\in I} d_i^{\varepsilon-1/2}
		}{\prod_{i\in I} d_i^{1/2+\varepsilon} \prod_{j\notin I} d_j} \\
		& \le x^{t(\varepsilon-1/2)} \sum_{d_1,\ldots,d_k=1}^{\infty} 
		\frac{|(\mu*\delta_G)(d_1,\ldots,d_k)|}{\prod_{i\in I} d_i^{1/2+\varepsilon} \prod_{j\notin I} d_j} \\
		& \ll x^{t(\varepsilon-1/2)},
	\end{align*}
	since the series is convergent (for $t\ge 1$). Using that
	$t(\varepsilon -1/2)< -1$ for $0<\varepsilon< (t-2)/(2t)$, here we
	need $t\ge 3$, we obtain $B\ll \frac1{x}$. 
	
	Case II. $t=1$: Let $d_{i_0}>x$, $d_i\le x$ for $i\ne i_0$, and
	consider a prime $p$.  If $p\mid d_i$ for some $i\ne i_0$, then $p\le x$. 
	If $p\mid d_{i_0}$ and $p>x$, then
	$p\nmid d_i$ for every $i\ne i_0$, and
	$(\mu *\delta_G)(d_1,\ldots,d_k)=0$, cf. Remark \ref{Remark_delta}. Hence
	it is enough to consider the primes $p\le x$. We deduce
	\begin{align*}
		B & < \frac1{x} \sum_{\substack{d_{i_0}> x\\ d_i\le x, i\ne i_0}}
		\frac{|(\mu* \delta_G)(d_1,\ldots,d_k)|}{\prod_{i\ne i_0} d_i}
		\\
		& \le \frac1{x} \prod_{p\le x} \sum_{\nu_1,\ldots, \nu_k=0}^{\infty}
		\frac{|(\mu* \delta_G)(p^{\nu_1},\ldots,p^{\nu_k})|}{p^{\sum_{i\ne i_0} \nu_i}}
		\\
		& \ll \frac1{x}(\log x)^{\max \deg(i_0)},
	\end{align*}
	similar to the estimate \eqref{Mertens}.
	
	Case III. $t=2$: Let $d_{i_0}>x$, $d_{i'_0}>x$. We split the sum $B$ into two sums, namely
	\begin{align*}
		B & = \sum_{\substack{d_{i_0}>x,d_{i'_0}>x\\ d_i\le x, i\ne i_0,i'_0}}
		\frac{|(\mu*\delta_G)(d_1,\ldots,d_k)|}{d_1\cdots d_k}
		\\
		& = \sum_{\substack{d_{i_0} > x^{3/2},d_{i'_0} >x\\ d_i\le x, i\ne i_0,i'_0}}
		\frac{|(\mu*\delta_G)(d_1,\ldots,d_k)|}{d_1\cdots d_k} +
		\sum_{\substack{x^{3/2}\ge d_{i_0}> x,d_{i'_0} >x\\ d_i\le x, i\ne i_0, i'_0}}
		\frac{|(\mu*\delta_G)(d_1,\ldots,d_k)|}{d_1\cdots d_k}
		\\
		& =: B_1+B_2,
	\end{align*}
	say, where
	\begin{align*}
		B_1 & = \sum_{\substack{d_{i_0}> x^{3/2},d_{i'_0} >x\\ d_i\le x, i\ne i_0,i'_0}}
		\frac{|(\mu*\delta_G)(d_1,\ldots,d_k)|}{d_{i_0}^{1/3} \prod_{i\ne i_0} d_i}
		\frac1{d_{i_0}^{2/3}}
		\\
		& <\frac1{x} \sum_{d_1,\ldots,d_k=1}^{\infty}
		\frac{|(\mu*\delta_G)(d_1,\ldots,d_k)|}{d_{i_0}^{1/3} \prod_{i\ne i_0} d_i} 
		\\
		& \ll \frac1{x},
	\end{align*}
	since the latter series is convergent. Furthermore,
	\begin{align*}
		B_2 <\frac1{x} \sum_{\substack{x^{3/2}\ge d_{i_0}, d_{i'_0} >x \\
				d_i\le x, i\ne i_0,i'_0}}
		\frac{|(\mu*\delta_G)(d_1,\ldots,d_k)|}{\prod_{i\ne i'_0} d_i},
	\end{align*}
	and consider a prime $p$. For the last sum, if $p\mid d_i$ for some $i\ne i'_0$
	then $p\le x^{3/2}$. If $p\mid d_{i'_0}$ and
	$p>x^{3/2}$, then $p\nmid d_i$ for every $i\ne i'_0$ and
	$(\mu*\delta_G)(d_1,\ldots,d_r)=0$, cf. Remark \ref{Remark_delta}. Hence it is enough
	to consider the primes $p\le x^{3/2}$. We deduce, similar to \eqref{A}, \eqref{Mertens} that
	\begin{align*}
		B_2 <\frac1{x} \prod_{p\le x^{3/2}} \sum_{\nu_1,\ldots, \nu_r=0}^{\infty}
		\frac{|(\mu*\delta_G)(p^{\nu_1},\ldots,p^{\nu_k})|}{p^{\sum_{i\ne i'_0} \nu_i}} 
		\ll \frac1{x}(\log x^{3/2})^{d_G} \ll \frac1{x}(\log x)^{d_G},
	\end{align*}
	with $d_G=\max_{i\in G} \deg(i)$.
	
	Hence given any $|I|=t\ge 1$ we have $B\ll \frac1{x}(\log x)^{d_G}$. Therefore, by
	\eqref{roro},
	\begin{equation} \label{second}
		\sum_{d_1, \ldots, d_r\le x}
		\frac{(\mu*\delta_G)(d_1,\ldots,d_k)}{d_1\cdots d_k} = A_G+
		O(x^{-1}(\log x)^{d_G}).
	\end{equation}
	
	The proof is complete by putting  together \eqref{main_term},
	\eqref{Mertens} and \eqref{second}. 
\end{proof}

\begin{proof}[Proof of Theorem {\rm \ref{Th_exact_least}}]
	According to identities \eqref{beta_r} and \eqref{beta_least_r} we have
	\begin{equation} \label{sum_exact}
		\sum_{n_1,\ldots,n_k\le x} \beta_r(n_1,\ldots,n_k)=
		\sum_{j=r}^{k(k-1)/2}  (-1)^{j-r} \binom{j}{r} \sum_{\substack{E\subseteq V^{(2)}\\ |E|=j}} \sum_{n_1,\ldots,n_k\le x}
		\delta_G(n_1,\ldots,n_k),    
	\end{equation}
	and
	\begin{equation} \label{sum_least}
		\sum_{n_1,\ldots,n_k\le x} \beta'_r(n_1,\ldots,n_k) = 
		\sum_{j=r}^{k(k-1)/2} (-1)^{j-r} \binom{j-1}{r-1} \sum_{\substack{E\subseteq S\\ |E|=j}} \sum_{n_1,\ldots,n_k\le x}
		\delta_G(n_1,\ldots,n_k).
	\end{equation}
	
	Now for the inner sums $\sum_{n_1,\ldots,n_k\le x}
	\delta_G(n_1,\ldots,n_k)$ of identities \eqref{sum_exact} and \eqref{sum_least} use asymptotic 
	formula \eqref{form_E}. For the complete coprimality graph with $E=V^{(2)}$, corresponding to all coprimality conditions, 
	the error term is $O(x^{k-1}(\log x)^{k-1})$, and this will be the final error term in both cases. This proves asymptotic 
	formulas \eqref{asympt_form_exact} and \eqref{asympt_form_least}. 
\end{proof}

\begin{proof}[Proof of Corollary {\rm \ref{Cor_asymt_beta_k}}]
	
	Apply formula \eqref{asympt_form_exact} for $r=0$, with the constant $C_{k,0}$ given by \eqref{C_k_r}.
\end{proof}

\section{Examples} \label{Section_Ex}

To illustrate identities \eqref{delta_Euler_prod} and \eqref{A_G_version_new} let us work out the following 
examples.

\begin{example} \label{Ex_1}
	{\rm  Let $k=4$ and $G=(V,E)$ with $V=\{1,2,3,4\}$, $E=\{\{1,2\},\{2,3\},
		\{3,4\}$, $\{4,1\}\}$, that is, $\gcd(n_1,n_2)=1$, $\gcd(n_2,n_3)=1$, $\gcd(n_3,n_4)=1$, 
		$\gcd(n_4,n_1)=1$. See Figure \ref{fig:1}.
		
		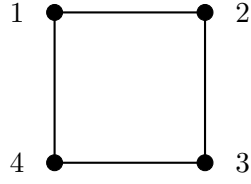
\begin{figure}[ht]
			\centering
			\begin{tikzpicture}
				\draw[fill=black] (-1,1) circle (3pt);
				\draw[fill=black] (1,1) circle (3pt);
				\draw[fill=black] (-1,-1) circle (3pt);
				\draw[fill=black] (1,-1) circle (3pt);
				\node at (-1.5,1) {1};
				\node at (1.5,1) {2};
				\node at (-1.5,-1) {4};
				\node at (1.5,-1) {3};
				\draw[thick] (-1,-1) -- (-1,1) -- (1,1) -- (1,-1);
				\draw[thick] (-1,-1) -- (1,-1);
			\end{tikzpicture}
			\caption{Graph of Example \ref{Ex_1}}
			\label{fig:1}
		\end{figure}
		
		Here $I=\{1,2,3,4\}$ and choose the minimum vertex cover $J=\{1,3\}$. 
		According to \eqref{delta_Euler_prod},
		\begin{equation} \label{delta_Euler_ex_1}
			D'_G(s_1,s_2,s_3,s_4) 
			= \prod_p \left(\sideset{}{'} \sum_{L\subseteq J} \prod_{\ell \in L} \frac1{p^{s_{\ell}}}  \prod_{i\in (J \setminus L) 
				\cup (N(L)\setminus J) } \left(1-\frac1{p^{s_i}}\right)\right)
		\end{equation}
		
		Write the terms of the sum in \eqref{delta_Euler_ex_1}, see Table \ref{tab:1},
		where $x_i=1/p^{s_i}$ ($1\le i\le 4$). Note that all subsets of $J$ are independent.
		
		\begin{table}[ht]
			\centering
			\begin{tabular}{|c|c|c|c|} \hline
				$L$   & $N(L)$ & $(J\setminus L) \cup (N(L)\setminus J)$ & $S_L$ \\ \hline\hline
				$\emptyset$   & $\emptyset$ & $\{1,3\}$ & $(1-x_1)(1-x_3)$ \\ \hline
				$\{1\}$   & $\{2,4\}$ & $\{2,3,4\}$ & $x_1(1-x_2)(1-x_3)(1-x_4)$ \\ \hline
				$\{3\}$   & $\{2,4\}$ &  $\{1,2,4\}$ & $x_3(1-x_1)(1-x_2)(1-x_4)$\\ \hline
				$\{1,3\}$   & $\{2,4\}$ & $\{2,4\}$ & $x_1x_3(1-x_2)(1-x_4)$\\ \hline
			\end{tabular}
			\caption{Terms of the sum in Example \ref{Ex_1}}
			\label{tab:1}
		\end{table}
		
		We obtain
		\begin{align*} 
			D'_G(s_1,\ldots,s_4) & =  \prod_p \left(S_{\emptyset}+S_{\{1\}}+S_{\{3\}}+ S_{\{1,3\}}  \right) \\
			& = \prod_p \left(1-\frac1{p^{s_1+s_2}}- \frac1{p^{s_1+s_4}} -\frac1{p^{s_2+s_3}} - \frac1{p^{s_3+s_4}} 
			+ \frac1{p^{s_1+s_2+s_3}} \right. \\
			& \left. + \frac1{p^{s_1+s_2+s_4}}+\frac1{p^{s_1+s_3+s_4}} + \frac1{p^{s_2+s_3+s_4}} 
			-  \frac1{p^{s_1+s_2+s_3+s_4}} \right).
		\end{align*}
		
		Observe that the terms $\pm 1/p^i$ with $i\in J=\{1,3\}$ cancel out, and we have the terms 
		$-1/p^{s_i + s_j}$ with $\{i,j\}\in E$, according to the edges of $G$. 
		Hence the infinite product is absolutely convergent provided that 
		$\Re (s_{i_1}+\cdots +s_{i_j})>1$ for every $i_1,\ldots,i_j\in \{1,2,3,4\}$ with $i_1< \cdots < i_j$, $2\le j\le 4$.
		
		The asymptotic density of $4$-tuples $(n_1,\ldots,n_4)\in \N^4$ such that $\gcd(n_i,n_j)=1$ with $\{i,j\}\in E$ is
		\begin{align*}
			D'_G(1\ldots,1)= \prod_p \left(1- \frac{4}{p^2}+\frac{4}{p^3} - \frac{1}{p^4} \right).
		\end{align*}
		
		This asymptotic density has been obtained using identity \eqref{A_G_version} by de Reyna and Heyman \cite[Sect.\ 4]{ReyHey2015}.
	}
\end{example}

\begin{example} \label{Ex_2}
	{\rm Now let $k=7$ and $G=(V,E)$ with $V=\{1,2,3,4,5,6,7\}$ and  
		$$E=\{\{1,2\},\{1,3\}, \{2,4\}, \{2,5\}, \{3,4\}, \{4,5\}\},
		$$
		that is, $\gcd(n_1,n_2)=1$, $\gcd(n_1,n_3)=1$, $\gcd(n_2,n_4)=1$, 
		$\gcd(n_2,n_5)=1$, $\gcd(n_3,n_4)=1$, $\gcd(n_4,n_5)=1$. See Figure \ref{fig:2}.
		
		\begin{figure}[ht]
			\centering
			\begin{tikzpicture}
				\draw[fill=black] (-1,1) circle (3pt);
				\draw[fill=black] (1,1) circle (3pt);
				\draw[fill=black] (-1,-1) circle (3pt);
				\draw[fill=black] (1,-1) circle (3pt);
				\draw[fill=black] (2.5,0) circle (3pt);
				\draw[fill=black] (4,0) circle (3pt);
				\draw[fill=black] (6,0) circle (3pt);
				\node at (-1.5,1) {1};
				\node at (1.5,1) {2};
				\node at (-1.5,-1) {3};
				\node at (1.5,-1) {4};
				\node at (3,0) {5};
				\node at (4.5,0) {6};
				\node at (6.5,0) {7};
				\draw[thick] (-1,-1) -- (-1,1) -- (1,1) -- (1,-1) -- (2.5,0);
				\draw[thick] (-1,-1) -- (1,-1);
				\draw[thick] (1,1) -- (2.5,0);
			\end{tikzpicture}
			\caption{Graph of Example \ref{Ex_2}}
			\label{fig:2}
		\end{figure}
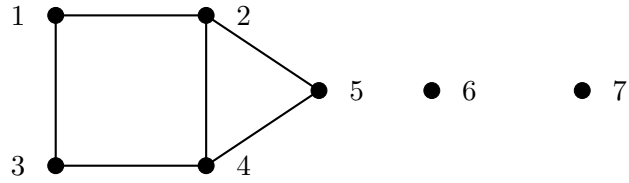
		\vskip1mm
		
		Here $I=\{1,2,3,4,5\}$, since the variables $n_6,n_7$ do not appear in the constraints. 
		Choose the minimum vertex cover $J=\{1,2,4\}$. 
		Consider the subsets $L$ of $J$
		and write the corresponding terms $S_L$ of the sum in \eqref{delta_Euler_prod}, see Table \ref{tab:2},
		where $x_i=p^{-s_i}$ ($1\le i\le 5$). The subsets $L=\{1,2\}$ and $L=\{2,4\}$ do not appear in the sum, 
		since $1,2$ and $2,4$ are adjacent vertices.
		
		\begin{table}[ht]
			\centering
			\begin{tabular}{|c|c|c|c|} \hline
				$L$   & $N(L)$ & $(J\setminus L) \cup (N(L)\setminus J)$ & $S_L$ \\ \hline\hline
				$\emptyset$   & $\emptyset$ & $\{1,2,4\}$ & $(1-x_1)(1-x_2)(1-x_4)$ \\ \hline
				$\{1\}$   & $\{3\}$ & $\{2,3,4\}$ & $x_1(1-x_2)(1-x_3)(1-x_4)$ \\ \hline
				$\{2\}$   & $\{5\}$ &  $\{1,4,5\}$ & $x_2(1-x_1)(1-x_4)(1-x_5)$\\ \hline
				$\{4\}$   & $\{3,5\}$ & $\{1,2,3,5\}$ & $x_4(1-x_1)(1-x_2)(1-x_3)(1-x_5)$\\ \hline
				$\{1,4\}$   & $\{3,5\}$ & $\{2,3,5\}$ & $x_1x_4(1-x_2)(1-x_3)(1-x_5)$\\ \hline
			\end{tabular}
			\caption{Terms of the sum in Example \ref{Ex_2}}
			\label{tab:2}
		\end{table}
		
		It follows that
		\begin{align*} 
			D'_G(s_1,\ldots,s_7)  & = \prod_p \left(S_{\emptyset}+S_{\{1\}}+S_{\{2\}}+ S_{\{4\}} + S_{\{1,4\}} \right) \\
			& = \prod_p \left(1-\frac1{p^{s_1+s_2}}- \frac1{p^{s_1+s_3}} -\frac1{p^{s_2+s_4}} - \frac1{p^{s_2+s_5}} -
			\frac1{p^{s_3+s_4}} - \frac1{p^{s_4+s_5}} \right. \\
			& \left. + \frac1{p^{s_1+s_2+s_3}}+ \frac1{p^{s_1+s_2+s_4}}+\frac1{p^{s_1+s_2+s_5}} + \frac1{p^{s_1+s_3+s_4}}+ \frac1{p^{s_2+s_3+s_4}} 
			\right.  \\
			& \left. +\frac{2}{p^{s_2+s_4+s_5}}  + \frac1{p^{s_3+s_4+s_5}} - \frac1{p^{s_1+s_2+s_3+s_4}}- \frac1{p^{s_1+s_2+s_4+s_5}} - \frac1{p^{s_2+s_3+s_4+s_5}} \right).
		\end{align*}
		
		Observe that the terms $\pm 1/p^i$ with $i,j\in \{1,2,4\}$ cancel out, and we have the terms 
		$-1/p^{s_i + s_j}$ with $\{i,j\}\in E$, according to the edges of $G$. 
		Here the infinite product is absolutely convergent provided that 
		$\Re (s_{i_1}+\cdots +s_{i_j})>1$ for every $i_1,\ldots,i_j\in \{1,2,3,4,5\}$ with $i_1< \cdots < i_j$, $2\le j\le 5$.
		
		The asymptotic density of $7$-tuples $(n_1,\ldots,n_7)\in \N^7$ with the corresponding constraints 
		$\gcd(n_i,n_j)=1$ with $\{i,j\}\in E$ is
		\begin{align*}
			D'_G(1\ldots,1)= \prod_p \left(1- \frac{6}{p^2}+\frac{8}{p^3} - \frac{3}{p^4} \right).
		\end{align*}
		
		Application of identity \eqref{A_G_version} by de Reyna and Heyman \cite{ReyHey2015} is more laborious here, 
		since $G$ has six edges and there are $2^6=64$ subsets of $E$.
	}
\end{example}

\begin{example} \label{Ex_3}
	{\rm  Now consider the case of pairwise coprime integers with $E=\{\{i,j\}: 1\le i<j\le k\}$.
		For $k=4$ the graph is in Figure \ref{fig:3}.
		
		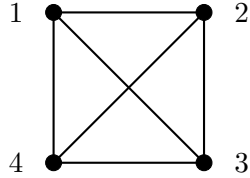
\begin{figure}[ht]
			\centering
			\begin{tikzpicture}
				\draw[fill=black] (-1,1) circle (3pt);
				\draw[fill=black] (1,1) circle (3pt);
				\draw[fill=black] (-1,-1) circle (3pt);
				\draw[fill=black] (1,-1) circle (3pt);
				\node at (-1.5,1) {1};
				\node at (1.5,1) {2};
				\node at (-1.5,-1) {4};
				\node at (1.5,-1) {3};
				\draw[thick] (-1,-1) -- (-1,1) -- (-1,1) -- (1,1) -- (1,-1);
				\draw[thick] (-1,-1) -- (1,-1);
				\draw[thick] (-1,-1) -- (1,1);
				\draw[thick] (-1,1) -- (1,-1);
			\end{tikzpicture}
			\caption{Graph to Example \ref{Ex_3}}
			\label{fig:3}
		\end{figure}
		\vskip1mm
		
		Here $I=\{1,\ldots,k\}$ and choose the minimum vertex cover $J=\{1,\ldots,k-1\}$. The only independent subsets $L$ 
		of $J$ are $L=\emptyset$ and $L=\{1\}$, ..., $L=\{k-1\}$ having one single element.  
		
		If $L=\emptyset$, then $N(L)=\emptyset$, $(J\setminus L) \cup (N(L)\setminus J)=J$ and obtain, with $x_i=p^{-s_i}$
		($1\le i\le k$),
		\begin{align*}
			S_{\emptyset}= (1-x_1)\cdots (1-x_{k-1}).
		\end{align*}
		
		If $L=\{\ell\}$, $\ell\in J$, then $N(L)=\{k\}$, $(J\setminus L) \cup (N(L)\setminus J) =\{1,\ldots,k\}\setminus \{\ell\}$, and have
		\begin{align*}
			S_{\{\ell\}} = x_{\ell} \prod_{\substack{j=1\\ j\ne \ell}}^k (1-x_j).
		\end{align*}
		
		We need to evaluate the sum
		\begin{equation} \label{sum}
			S:= S_{\emptyset}+\sum_{\ell=1}^{k-1} S_{\{\ell \}}. 
		\end{equation}
		
		Let $e_j(x_1,\dots,x_k)=\sum_{1\le i_1<\ldots<i_j\le k}
		x_{i_1}\cdots x_{i_j}$ denote the elementary symmetric polynomials
		in $x_1,\ldots,x_k$ of degree $j$ ($j\ge 0$). By convention,
		$e_0(x_1,\ldots,x_k)=1$.
		
		Consider the polynomial
		\begin{align*} 
			P(x) =\prod_{j=1}^k (x-x_j) = \sum_{j=0}^k (-1)^j e_j(x_1,\dots,x_k) x^{k-j}.
		\end{align*}
		
		Its derivative is
		\begin{align*} 
			P'(x)= \sum_{j=0}^{k-1} (-1)^j (k-j) e_j(x_1,\dots,x_k)
			x^{k-j-1},
		\end{align*}
		and on the other hand
		\begin{align*} 
			P'(x)= \sum_{j=1}^k \prod_{\substack{i=1\\ i\ne j}}^k (x-x_i).
		\end{align*} 
		
		We obtain that the sum \eqref{sum} is
		\begin{align*}
			S & =\prod_{j=1}^{k-1} (1-x_j) + \sum_{j=1}^{k-1} x_j \prod_{\substack{i=1\\ i\ne j}}^k (1-x_i) \\
			& = \sum_{j=1}^k \prod_{\substack{i=1\\ i\ne j}}^k (1-x_i) -(k-1) \prod_{j=1}^k (1-x_j)\\
			& = P'(1) -(k-1)P(1) \\
			& = 1+ \sum_{j=2}^k (-1)^{j-1} (j-1) e_j(x_1,\ldots,x_k), 
		\end{align*}
		that is,
		\begin{align*}
			\sum_{\substack{n_1,\ldots,n_k=1\\ \gcd(n_i,n_j)=1, \, 1\le i<j\le k}}^{\infty} \frac1{n_1^{s_1}\cdots n_k^{s_k}} 
			= \prod_p \left(1+ \sum_{j=2}^k (-1)^{j-1} (j-1) e_j(p^{-s_1},\ldots,p^{-s_k})\right). 
		\end{align*}
		
		For $s_1=\cdots = s_k=1$ this gives identity \eqref{density_pairwise_rel_prime}, representing the asymptotic density of $k$-tuples  with pairwise relatively prime components.
	}
\end{example}

\end{document}